\documentclass[11pt]{amsart}
\usepackage{epsfig,graphicx, subfigure,epstopdf}
\usepackage{amsmath, amssymb,mathabx,mathrsfs}
\usepackage{color}
\usepackage[toc,page]{appendix}
\usepackage{comment}

\newtheorem{thm}{Theorem}[section]

\newtheorem{prop}[thm]{Proposition}
\theoremstyle{definition}

\theoremstyle{remark}
\newtheorem{rem}[thm]{Remark}

\numberwithin{equation}{section}

\theoremstyle{definition}

\newcommand{\inte}{\textrm{int}}

\begin{document}
	\title[Simple proof of anassov's conjecture]{A proof of Atanassov's Conjecture and other generalizations of Sperner's Lemma.}

	\author{Yitzchak Shmalo}
	\address{}
	\email{ }

	\begin{abstract}
		A simple proof of Atanassov's Conjecture is presented. Atanassov's Conjecture is a generalization of Sperner's Lemma, a lemma which has been used to prove Brouwer's Fixed Point Theorem, among other fixed point theorems. The proof of Atanassov's Conjecture is based on the Brouwer Degree of maps and is extremely elementary. It is much simpler than the original proofs given for the conjecture and provides some insight into the nature of the conjecture. Furthermore, a generalization of the conjecture is presented and finally a new theorem, similar to the original Sperner Lemma, is proved.             
	\end{abstract}

	\maketitle

	\section{Introduction}
	
	In the beginning of the twentieth century, Luitzen Brouwer proved the Brouwer Fixed Point Theorem, which states that every continuous function $f$ from an $n$-dimensional ball to itself has a fixed point, meaning that for some $x$ we have $f(x)=x$. Historically, Sperner's Lemma has been used to provide a straightforward proof for Brouwer's Fixed Point Theorem. Some have also used Sperner's Lemma to prove Kakutani Fixed Point Theorem, a generalization of Brouwer's Fixed Point Theorem. Such proofs are generally simpler than those previously offered for the existence of fixed points and are also somewhat constructive, meaning that they can be used to approximate the position of fixed points. Sperner's Lemma itself is also simple to prove and much is known about generalizations of the lemma, for example Atanassov's Conjecture. In this paper, proofs for Atanassov's Conjecture and other propositions similar to Sperner's Lemma are given.

	\subsection{Sperner's Lemma for simplices}\label{sec:sperner}
	Consider an $n$-simplex $T$, and $T=\bigcup_{i\in I} T_i$, with $I$ finite, a simplicial decomposition of $T$.
	A labeling of $T$ is a map $\phi:T\to\{1,2,\ldots,n+1\}$. In particular, each vertex $v$ of $T$ and of any of the $T_i$'s is assigned a unique label $\phi(v)\in \{1,2,\ldots,n+1\}$.
	
	A simplex $T$ is said to be \emph{completely labeled} if its vertices are assigned all labels from $\{1,2,\ldots,n+1\}$.
	
	The labeling $\phi$ is called a \emph{Sperner labeling} if every  point $p$ that lies on some face $S$ of $T$  is assigned one of the labels of the vertices of $S$.
	There are no restrictions on the labeling of the interior points of $T$.
	
	In the most basic form the Sperner Lemma states the following:
	\begin{thm}[Sperner Lemma \cite{Sperner1928}]
		Given an $n$-simplex $T$ with a complete Sperner labeling, a simplicial decomposition $T=\bigcup_{i} T_i$.
		Then the number of completely labeled simplices $T_i$ is odd.
		In particular,  there exists at least one simplex $T_i$ that is completely labeled.
	\end{thm}
	
	Atanassov's Conjecture is a generalization of the lemma which states the following.
	
	A simplex $P_i$ is completely labeled if it has vertices carrying $d+1$ different labels. Let $P$ be a convex, $d$-dimensional polytope, with $n$ vertices, which is divided into convex, $d$-dimensional simplices $\{P_i\}_{i\in I}$, with $I$ finite, such that $P=\bigcup_{i\in I}P_i$, and  for $i\neq j$, $\inte (P_i)\cap \inte (P_j)=\emptyset$ and $P_i\cap P_j$ is either empty or a face of both $P_i$ and $P_j$. Assume that $\phi:V(P)\cup\bigcup_{i\in I}V(P_i)\to\{1,2,\ldots,n\}$ is a labeling of the vertices of $P$ and of $P_i$ such that:
	\begin{itemize}
		\item  Each label from $\{1,2,\ldots,n\}$ is assigned to some  vertex of $P$, so  no two vertices are assigned the same label;
		\item If a vertex   of $P_i$ lies on some face $W$ of $P$, then that vertex is assigned one of the labels of  the vertices of $W$.
	\end{itemize}
	
	Then there are at least $n-d$ simplices of $\{P_i\}_{i\in I}$; each completely labeled. This theorem was first proved in \cite{Su2002}.

	\section{Preliminaries}

	\subsection{Sperner's Lemma for polytopes}\label{sec:bekker}
	First we describe a more general Sperner Lemma-type of result for polytopes, following \cite{Bekker1995}. Parts of this section were taken almost directly from \cite{GS}. Let $P$ be a convex $n$-dimensional  polytope. We consider a labeling $\phi:P\to \{1,2,\ldots, n+1\}$ of $P$.
	
	A polytope is said to be \emph{completely labeled} if  its vertices are assigned all labels from $\{1, 2, \ldots, n+1\}$.
	
	A labeling $\phi:P\to \{1,\ldots,n+1\}$ is said to be a \emph{non-degenerate labeling} if no $(n-1)$-dimensional face of $P$ contains points which take $(n+1)$ or more different values. Thus, every $n$-dimensional face of $P$ can only contains points which take $(n+1)$ or less different labels.
	
	We  introduce some tools.
	
	Consider the standard $n$-dimensional simplex $T=\textrm{conv}(0, e_1,e_2,\ldots ,e_n)\subset \mathbb{R}^n$, where we denote by $(e_1,e_2,\ldots ,e_n)$ the standard basis of $\mathbb{R}^n$, and by $\textrm{conv}(\cdot)$ the convex hull of a set. We denote the corresponding vertices of $T$ as $\{a_1,\ldots, a_{n+1}\}$.   
	
	Given an $n$-dimensional, oriented, convex polytope  $P$, a  labeling  $\phi:P\to\{1,\ldots,n+1\}$, and the standard $n$-simplex $T$ of vertices $a_1,\ldots,a_{n+1}$, a \emph{realization} of $\phi$ is a continuous map $f: P \to T$, satisfying the following conditions:
	\begin{itemize}
		\item[(i)] If $v$ is a vertex of $P$ then $f(v)=a_{\phi(v)}$, i.e., $f(v)$ is the vertex $a_i$ of  $T$ with the index $i$ equal to the label of $v$;
		\item[(ii)] If $S$ is face of $P$  with vertices $v_1,\ldots,v_k$,  then $f(S)\subset \textrm{conv}(a_{\phi(v_1)},\ldots, a_{\phi(v_k)})$.
	\end{itemize}
	Informally, a realization of $P$ is a continuous mapping of $P$ onto $T$ that `wraps' $\partial P$ around $\partial T$, such that the labels of the vertices of $P$ match with the indices $i$ of the vertices of $T$. Such $f$ is in general non-injective. For a smooth, boundary preserving map $f:(M,\partial M)\to (N,\partial N)$ between two oriented $n$-dimensional manifolds with boundary, it is known that $\deg(f)=\deg(\partial f)$, where  $\partial f:\partial  M\to\partial N$ is the map induced by $f$ on the boundaries. The degree of the map $f$ can be defined as the signed number of preimages $f^{-1}(p)=\{q_1,\ldots,q_k\}$ of a regular value  $p$ of the map $f$, where each point $q_i$ is counted with a sign $\pm 1$ depending on whether  $df_{q_i}:T_{q_i}M\to T_{p}N$ is orientation preserving or orientation reversing. That is, $\deg(f)=\sum_{q\in f^{-1}(p)}\textrm{sign} (\det (df_{q}))$, where $p\in  N\setminus \partial N$ is a regular value of $f$. The definition of the Brouwer degree extends via homotopy to continuous maps.

	\begin{prop}[\cite{Bekker1995}]\label{prop:Bekker}  Let $P$  be a convex  $n$-dimensional polytope.
		\begin{itemize}
			\item[(i)] Any labeling $\phi$ of some $P$ admits a realization $f$;
			\item[(ii)] Any two realizations of the same labeling are homotopic as maps of pairs $(P,\partial P)\mapsto (T,\partial T)$;
			\item[(iv)] If $\deg(f)\neq 0$ then $P$ is completely labeled, where $\deg(f)$ is the Brouwer Degree of the realization of $\phi$, a labeling of $P$.
		\end{itemize}
	\end{prop}

	The following is a generalization of the Sperner Lemma from  \cite{Bekker1995}.

	\begin{thm}[\cite{Bekker1995}]\label{thm:bekker2}
		Assume that $P$ is an $n$-dimensional polytope,  $P=\bigcup_{i\in I}P_i$ is a decomposition of $P$ into polytopes as above, and $\phi:P\to\{1,\ldots,n+1\}$ is a \emph{non-degenerate labeling}.   If $\deg(\partial f)\neq 0$, then there exists a polytope $P_i$ that is completely labeled.
	\end{thm}
	
	\begin{proof}
		Let $f:P\to T$ be a realization of the labeling $\phi$; the existence of $f$ is ensured by Proposition \ref{prop:Bekker}.
		We transform the polytope $P$ into another polytope $P^*$, homotopically transform the  decomposition  $P=\bigcup_{i\in I}P_i$ into another decomposition  $P^*=\bigcup_{i\in I}P^*_i$, and homotopically transform the map
		$f:P\to T$ into a map $f^*:P^*\to T$ as follows:
		\begin{itemize}
			\item We construct a new  polytope $P^*$ by appending  to the  vertices of $P$ all the vertices of the the $P_i$'s lying on the faces of $P$, and appending to the faces of $P$ all the faces of the $P_i$'s lying on the faces of $P$; the resulting polytope $P^*$  still has the decomposition $P^*=\bigcup_{i\in I}P_i$;
			\item We apply a homotopy deformation to the decomposition $\bigcup_{i\in I}P_i$ and to the realization $f$ to obtain a new decomposition  $P^*=\bigcup_{i\in I}P^*_i$, with the labeling of $P^*_i$ inherited from that of $P_i$, and a new realization  $f^*:P^*\to T$ so that, for every $i$, the restriction of $f^*_{\mid {P^*_i}}$ to $P^*_i$ maps $\partial P^*_i$ to $\partial T$ and is also a realization. The later property ensures that $\deg(f^*_{\mid {P^*_i}})=\deg(\partial f^*_{\mid {\partial P^*_i}})$.
		\end{itemize}
		
		From the condition that  if a vertex   of $P^*_i$ lies on some face $S$ of $P^*$, then that vertex is assigned one of the labels of  the vertices of $S$, we have $\deg(f^*) = \deg(f)$.
		Using the addition and homotopy properties of the Brouwer degree, we obtain that
		\[\deg(\partial f) = \deg(f) = \deg(f^*) = \sum_i \deg(f^*_i) =  \sum_i \deg( \partial f^*_i)\]

		If $\deg(f)\neq 0$, then $\sum_i \deg( \partial f^*_i)\neq 0$, which means that  there exists a polytope $P^*_i$ such that $\deg(\partial f^*_{\mid \partial P^*_i})\neq 0$, hence $\deg(f^*_{\mid P^*_i})\neq 0$. By Proposition \ref{prop:Bekker}, and since the labeling of $P_i$ is the same as the labeling of $P^*_i$, we have that $P_i$ is completely labeled.
	\end{proof}

	Now we give the notion of a cover of a polytope, following \cite{Su2002}. A cover $C$ of a convex polytope $P$ is a collection of simplices in $P$ such that $\bigcup_{S \in C} S = P$. We say that a labeled set of simplices cover a polytope  $P$, or is a cover of $P$, under a map $f$, if the following conditions hold: 
	\begin{itemize}
		\item[(i)] If $v$ is a vertex of $S$ in the collection then $f(v)=a_{\phi(v)}$, i.e., $f(v)$ is the vertex $a_i$ of  $P$ with the index $i$ equal to the label of $v$;
		\item[(ii)] If $R$ is face of $S$  with vertices $v_1,\ldots,v_k$,  then $f(R)\subset \textrm{conv}(a_{\phi(v_1)},\ldots, a_{\phi(v_k)})$.
		\item[(ii)] $f$ is surjective.
	\end{itemize}
	
	\begin{thm}[\cite{Su2002} and \cite{Rothschild}]\label{thm:bekker2}
		Let $C(P)$ denote the covering number of an $(n,p)$-polytope $P$;
		which is the size of the smallest cover of $P$. Then, $C(P) \geq (n-d)$. This result is best possible as the equality is attained for stacked polytopes.
	\end{thm}

	\section{Proof of Atanassov's Conjecture}
	
	\begin{thm} Let $P$ be a convex, $d$-dimensional polytope, with $n$ vertices, which is divided into convex, $d$-dimensional simplices $\{P_i\}_{i\in I}$, with $I$ finite, such that $P=\bigcup_{i\in I}P_i$, and  for $i\neq j$, $\inte (P_i)\cap \inte (P_j)=\emptyset$ and $P_i\cap P_j$ is either empty or a face of both $P_i$ and $P_j$. Assume that $\phi:V(P)\cup\bigcup_{i\in I}V(P_i)\to\{1,2,\ldots,n\}$ is a labeling of the vertices of $P$ and of $P_i$ such that:
		\begin{itemize}
			\item  Each label from $\{1,2,\ldots,n\}$ is assigned to some  vertex of $P$, so  no two vertices are assigned the same label;
			\item If a vertex   of $P_i$ lies on some face $W$ of $P$, then that vertex is assigned one of the labels of  the vertices of $W$.
		\end{itemize}
		
		A simplex $P_i$ is completely labeled if it has vertices carrying $d+1$ different labels. Atanassov's Conjecture states that there are at least $n-d$ simplices $P_i$; each completely labeled.
		
	\end{thm} 
	
	\begin{proof} 
		
		Denote by $a_i$ the vertex of $P$ labeled with $i$. Form a realization of $\phi$ as a continuous map $f: P \to P$, satisfying the following conditions:
		\begin{itemize}
			\item If $v$ is a vertex of $P$ then $f(v)=a_{\phi(v)}$, i.e., $f(v)$ is the vertex of $P$ of index equal to the label of $v$;
			\item If $R$ is face of $P$ with vertices $w_1,\ldots,w_k$ then $f(R)\subset \textrm{conv}(a_{\phi(w_1)},\ldots, a_{\phi(w_k)})$, where $\textrm{conv}(\cdot)$ denotes the convex hull of a set.
			\item For any $i$, the restriction of $f|_i$ to $P_i$, with $P_i$ having vertices $l_1,\ldots,l_k$, is such that $f(P_i) \subset \textrm{conv}(a_{\phi(l_1)},\ldots, a_{\phi(l_k)})$ and is also a realization.
		\end{itemize}

		Given that a vertex of $P_i$ which lies on some face $W$ of $P$ is assigned one of the labels of  the vertices of $W$, we have that this labeling is \emph{non-degenerate}. From the conditions of the realization and the first condition of the labeling, we have that $deg(\partial f) = deg(f) = 1$. Thus, $f$ is surjective, which means that a set of $P_i$ simplices cover $P$ under $f$. Notice, however, that every $P_i$ which is not completely labeled is mapped to a $n-1$ or smaller dimension of $P$, under the realization, and cannot be used to cover $P$. So only completely labeled $P_i$ polytopes can be used to cover $P$. But the smallest number of  simplices needed to cover a $d$-dimensional polytope with $n$ vertices is $n-d$. Thus we must have at least $n-d$ complete simplices $P_i$. $P$ must be split into at least $n-d$ different simplices and each one must be mapped to by at least one $P_i$ and each such $P_i$ is completely labeled.   
		
	\end{proof}
	
	\begin{rem}
		After writing this paper I found a similar proof in \cite{Musin}, one which also uses the notion of a degree. Nevertheless, the use of the equality $\deg(\partial f) = \deg(f)$ greatly simplifies the proof.   
	\end{rem}

	The following is a generalizations of Atanassov's conjecture.
	
	\begin{thm}\label{generlization 1} Let $P$ be a convex, $d$-dimensional polytope, with $n$ vertices,  $P=\bigcup_{i\in I}P_i$ is a decomposition of $P$ into simplices as above. Assume that $\phi:V(P)\cup\bigcup_{i\in I}V(P_i)\to\{1,2,\ldots,n\}$ is a \emph{non-degenerate} labeling of the vertices of $P$ and of $P_i$. A simplex $P_i$ is completely labeled if it has vertices carrying $d+1$ different labels. Take $P'$ to be a $d$-dimensional polytope, with $n$ vertices. We require that each label from $\{1,2,\ldots,n\}$ is assigned to some  vertex of $P'$, so  no two vertices are assigned the same label. Take $f$ to be a realization of $\phi$, where $f: P \to P'$. Then we have at least $((n-d))(\deg(\partial f))$ different completely labeled simplices $P_i$. Furthermore, if $\deg(\partial f)=0$ then we will have an even number of completely labeled simplices.      
		
	\end{thm} 
	
	\begin{proof}
		Follows almost immediately from the previous proof.	
	\end{proof}
	
	\section{Neighboring Labeling Lemma}
	
	In this section a new labeling lemma is presented, one very similar to the original Sperner Lemma. The concept of the Brouwer Degree is again used.   
	
\begin{thm}
	Let $P$ be a convex, $d$-dimensional polytope, with $n$ vertices, which is divided into convex, $d$-dimensional polytopes $\{P_i\}_{i\in I}$, with $I$ finite, such that $P=\bigcup_{i\in I}P_i$, and  for $i\neq j$, $\inte (P_i)\cap \inte (P_j)=\emptyset$ and $P_i\cap P_j$ is either empty or a face of both $P_i$ and $P_j$. Assume that $\phi:V(P)\cup\bigcup_{i\in I}V(P_i)\to\{1,2,\ldots,n\}$ is a labeling of the vertices of $P$ and of $P_i$ such that:
\begin{itemize}
	\item  Each label from $\{1,2,\ldots,n\}$ is assigned to some  vertex of $P$, so no two vertices are assigned the same label. We call two labels $d$-similar if there is a $d$ dimensional face of $P$ which they share. Also, a label is $n$-similar to itself for all $n$.
	\item Every vertex $v$ of every $P_i$ is assigned one of the labels of the vertices of $P$ in the $car(v)$.
	\item If a vertex of $P_i$  shares a $k$-face with another vertex (either of $P_i$ or of $P_j$) then both of their labels must be $k$-similar, for all $k$.
\end{itemize}

A polytope $P_i$ is completely labeled if it has vertices carrying labels $\{1,2,\ldots, n\}$. This lemma states that there is at least one completely labeled $P_i$.

\end{thm}
	
	\begin{proof}
		We form a realization $f: P \to P$ of $\phi$. Any $P_i$ is either mapped only to $\partial P$ or covers $P$ entirely. This observation follows from the condition that if a vertex of $P_i$ is connected to another vertex (either of $P_i$ or of $P_j$) then both of their labels must be similar. Once again we have that $deg(\partial f)  =   \deg(f) = 1$. But if $\deg(f) = 1$ we must have at least one complete $P_i$, for any individual $P_i$ either covers $P$ entirely, in which case it is completely labeled, or only covers the boundary of $P$. But if all of the $P_i$'s only cover the boundary of $P$ then $   \deg(f) \not = 1$
		
	\end{proof}

	\section*{Acknowledgement}
	
	The author is grateful to Marian Gidea, who read and commented on the first draft of this work. As mentioned, parts of the second section were taken from a paper which we both authored.

	\bibliographystyle{alpha}
	
	\bibliography{Sperner_biblio}

\end{document}